\documentclass[10pt,a4paper,oneside]{amsart}

\usepackage[T1]{fontenc}
\usepackage[margin=1in]{geometry}

\usepackage{amsmath,amsfonts,amssymb,amsthm,url,verbatim}
\usepackage[dvips]{graphicx}
\usepackage[croatian,english]{babel}
\usepackage{amsfonts}
\usepackage[utf8]{inputenc}
\usepackage[T1]{fontenc}
\usepackage{amsmath}
\usepackage{array}
\usepackage{amssymb}
\usepackage{blindtext}
\usepackage{mathtools}
\usepackage{latexsym}
\usepackage[usenames,dvipsnames]{xcolor}
\usepackage{soul}
\usepackage[pagebackref]{hyperref}
\usepackage{color}
\usepackage{cleveref}
\usepackage[dvips]{graphicx}
\usepackage[croatian,english]{babel}
\usepackage{amsfonts}
\usepackage[utf8]{inputenc}
\usepackage[T1]{fontenc}
\usepackage[thinlines]{easytable}
\usepackage{ctable}
\usepackage{enumitem}
\usepackage{makecell}
\usepackage[dvipsnames]{xcolor}

\newtheorem{theorem}{Theorem}[section]

\newtheorem{proposition}[theorem]{Proposition}
\newtheorem{example}[theorem]{Example}
\newtheorem{lemma}[theorem]{Lemma}
\newtheorem{corollary}[theorem]{Corollary}

\theoremstyle{definition}
\newtheorem{remark}{Remark}
\newtheorem*{proposition*}{Proposition}

\newcommand{\ord}{\text{ord}}

\newcommand{\Q}{\mathbb{Q}}

\newcommand{\Z}{\mathbb{Z}}
\newcommand{\F}{\mathbb{F}}
\newcommand{\FF}{\mathcal{F}}

\newcommand{\Gal}{\operatorname{Gal}}

\DeclareMathOperator{\res}{res}
\newcommand{\lmfdbec}[2]{\href{http://www.lmfdb.org/EllipticCurve/Q/#1}{#2}}

\title{Tamagawa Numbers of Elliptic Curves with Prescribed Torsion Subgroup or Isogeny}
\date{\today}
\author{Antonela Trbović}
\address{Department of Mathematics, Faculty of Science, University of Zagreb, Bijenička cesta 30, 10000 Zagreb, Croatia}
\email{antonela.trbovic@math.hr}
\thanks{The author was supported by the QuantiXLie Centre of Excellence, a project
	co-financed by the Croatian Government and European Union through the
	European Regional Development Fund - the Competitiveness and Cohesion
	Operational Programme (Grant KK.01.1.1.01.0004) and by the Croatian Science Foundation under the project no. IP-2018-01-1313.}

\begin{document}
\maketitle

\begin{abstract}
We study Tamagawa numbers of elliptic curves with torsion $\Z/2\Z\oplus \Z/14\Z$ over cubic fields and of elliptic curves with an $n-$isogeny over $\Q$, for $n\in\{6,8,10,12,14,16,17,18,19,37,43,67,163\}$. 
Bruin and Najman \cite{BN} proved that every elliptic curve with torsion $\Z/2\Z \oplus \Z/14\Z$ over a cubic field is a base change of an elliptic curve defined over $\Q.$ We find that Tamagawa numbers of elliptic curves defined over $\Q$ with torsion $\Z/2\Z\oplus \Z/14\Z$ over a cubic field are always divisible by $14^2$, with each factor $14$ coming from a rational prime with split multiplicative reduction of type $I_{14k},$ one of which is always $p=2.$ The only exception is the curve \lmfdbec{1922c1}{1922.e2}, with $c_E=c_2=14.$
The same curves defined over cubic fields over which they have torsion subgroup $\Z/2\Z \oplus \Z/14\Z$ turn out to have the Tamagawa number divisible by $14^3$. As for $n-$isogenies, Tamagawa numbers of elliptic curves with an $18-$isogeny must be divisible by 4, while elliptic curves with an $n-$isogeny for the remaining $n$ from the mentioned set must have Tamagawa numbers divisible by 2, except for finite sets of specified curves. 
\end{abstract}

\section{Introduction}

Let $E$ be an elliptic curve over a number field $K$ and denote by $\Sigma$ the set of all finite primes of $K$. For each $v\in\Sigma,$ the completion of $K$ at $v$ will be denoted by $K_v$ and the residue field of $v$ by $k_v=\mathcal{O}_{K_v}/(\pi)$, where $\mathcal{O}_{K_v}$ is the ring of integers of $K_v$ and $\pi$ is a uniformizer of $\mathcal{O}_{K_v}.$

The subgroup $E_0(K_v)$ of $E(K_v)$ consists of all the points that reduce modulo $\pi$ to a non-singular point of $E(k_v)$. It is known that this group has finite index in $E(K_v)$ so we can define the Tamagawa number $c_v$ of $E$ at $v$ to be that index, i.e., $$c_v:=\left[E(K_v) : E_0(K_v)\right].$$ In light of this, we define the Tamagawa number of $E$ over $K$ to be the product
$c_{E/K}:=\prod_{v\in\Sigma}c_v.$ We will write $c_E$ instead of $c_{E/K}$ wherever it does not cause confusion.

It makes sense to study how the value $c_E$ depends on $E(K)_{tors},$ since $c_E/\#E(K)_{tors}$ appears as a factor in the leading term of the $L-$function of $E/K$ in the conjecture of Birch and Swinnerton-Dyer (see, for example, \cite[Conj. F.4.1.6]{HS}).

We start with some known results about Tamagawa numbers, first of which is given by Lorenzini in \cite{DL} on Tamagawa numbers of elliptic curves defined over $\Q$ with a specific torsion subgroup.
\begin{proposition}{\cite[Proposition 1.1]{DL}}\label{lor}
	Let $E/\Q$ be an elliptic curve with a $\Q-$rational point of order $N$. The following statements hold with at most five explicit exceptions for a given $N$.
	\begin{itemize}
		\item[(a)] If $N=4$, then $(N/2)\mid c_E$, except for $E=\lmfdbec{15a8}{X_1(15)}, \lmfdbec{15a7}{15.a4},$ and $ \lmfdbec{17a4}{17.a4}.$
		\item[(b)] If $N=5,6$ or $12$, then $N\mid c_E$, except for $E=\lmfdbec{11a3}{X_1(11)}, \lmfdbec{14a4}{X_1(14)}, \lmfdbec{14a6}{14.a4},$ and $ \lmfdbec{20a2}{20.a3}.$
		\item[(c)] If $N=10$, then $(N^2/2)\mid c_E$.
		\item[(d)] If $N=7,8$ or $9$, then $N^2\mid c_E$, except for $E=\lmfdbec{15a4}{15.a8}, \lmfdbec{21a3}{21.a3}, \lmfdbec{26b1}{26.b2}, \lmfdbec{42a1}{42.a5}, \lmfdbec{48a6}{48.a6}, \lmfdbec{54b3}{54.b2}$ and $ \lmfdbec{102b1}{102.c5}.$
	\end{itemize}
Without exception, $N\mid c$ if $N=7,8,9,10$ or $12$.
\end{proposition}
He also proved that the smallest possible ratio $c_E/\#E(K)_{tors}$ for elliptic curves over $\Q$ is $1/5$, achieved only by the modular curve \lmfdbec{11a3}{$X_1(11)$}. He gave as well some results about Tamagawa numbers of elliptic curves over quadratic extentions. Some of his results mentioned in Proposition \ref{lor} were later expanded upon by Krumm in \cite{K} which are presented with the following result.

\begin{proposition}{\cite[Propositions 5.2.2, 5.2.3]{K}}\label{kru}
Let $E/\Q$ be an elliptic curve with a $\Q-$rational point of order $N$.
\begin{itemize}
	\item[(a)] If $N=7,$ then $7\mid c_2.$
	\item[(b)] If $N=9,$ then $9\mid c_2$ and $3\mid c_3.$
\end{itemize}
\end{proposition}
Furthermore, Krumm proved some results on Tamagawa numbers of elliptic curves with prescribed torsion over number fields of degree up to 5. He also conjectured that $\ord_{13}(c_E)$ is even for all elliptic curves defined over quadratic fields with a point of order 13 and the same conjecture was later proved by Najman in \cite{N}. In their recent work \cite{BR}, Barrios and Roy explicitly classified Tamagawa numbers of elliptic curves defined over $\Q$ with non-trivial torsion subgroups at primes of additive reduction.

The results mentioned in Prosopitions \ref{lor} and \ref{kru} were the motivation to explore further the problem of finding Tamagawa numbers of elliptic curves with a certain torsion subgroup and at a certain prime. Bruin and Najman \cite{BN} proved that every elliptic curve with torsion $\Z/2\Z \oplus \Z/14\Z$ over a cubic field is a base change of an elliptic curve defined over $\Q.$ Using that fact, we prove in Section \ref{tors} that the Tamagawa numbers of all elliptic curves defined over $\Q$ that have torsion subgroup $\Z/2\Z\oplus \Z/14\Z$ over a cubic field are always divisible by $14^2$, except in the case of the curve \lmfdbec{1922c1}{1922.e2}, where $c_E=c_2=14.$ For each such curve we prove that at $p=2$ the reduction is split multiplicative, so $c_2=14k,$ and there always exists one more prime, distinct from 2, at which the reduction is also split multiplicative of type $I_{14t}.$ As a consequence of this result we get that elliptic curves defined over a cubic field with torsion subgroup $\Z/2\Z\oplus \Z/14\Z$ have Tamagawa numbers divisible by $14^3.$ We mention the explicit results in the following theorem.

\begin{theorem}
	\begin{itemize}
		\item Let $E$ be an elliptic curve defined over $\Q$ with torsion subgroup $\Z/2\Z \oplus \Z/14\Z$ over a cubic field. 
		\begin{itemize}
			\item[(a)] The reduction at $2$ is split multiplicative of type $I_{14k}$ and $c_2=14k,$ where $k\in\Z, \: k\geq 1.$
			\item[(b)] There exist at least 2 rational primes with split multiplicative reduction of type $I_{14k},$ where $k\in\Z, \: k\geq 1,$ one of which is always the prime $2$, so $14^2 \mid c_E,$ except for the curve $\lmfdbec{1922c1}{1922.e2}$, where $c_E=c_2=14.$
		\end{itemize}
	\item Let $E$ be an elliptic curve defined over a cubic field $K$ with torsion subgroup $\Z/2\Z \oplus \Z/14\Z$. Let $\mathfrak{P}$ be a prime of $K$ over $2$. Then the reduction at $\mathfrak{P}$ is split multiplicative of type $I_{14k}$ and $c_\mathfrak{P}=14k,$ where $k\in\Z, \: k\geq 1.$ Furthermore, $14^3 \mid c_E.$
	\end{itemize}
\end{theorem}

The proof of this theorem is given by the proofs of Propositions \ref{prop:p=2}, \ref{main} and Corollary \ref{kor} in Section \ref{tors}.

The question which naturally appears next is how does the Tamagawa number of an elliptic curve depend on the isogenies of that elliptic curve. In Section \ref{iso} we give a series of propositions which gives us first results about Tamagawa numbers of elliptic curves with prescribed isogeny. For elliptic curves defined over $\Q$, we were able to prove that if an elliptic curve has an $18-$isogeny, then its Tamagawa number is always divisible by $4$, and if it has an $n-$isogeny, for $n\in \{6,8,10,12,14,16,17,18,19,37,43,67,163\}$, then it has to be divisible by 2. There are finitely many exceptions for some of these results, all of which we list and give their Tamagawa numbers. 

\begin{theorem}
	Let $E$ be an elliptic curve over $\Q$ with an $N-$isogeny.
\begin{itemize}
	\item If $N=18,$ then $4|c_E,$ except for the curves $\lmfdbec{14a3}{14.a2}, \lmfdbec{14a4}{14.a5}, \lmfdbec{14a5}{14.a1}, \lmfdbec{14a6}{14.a4}$, where $ c_E=2$.
	\item If $N=10,$ then $2|c_E.$
	\item If $N=8,$ then $2|c_E$, except for the curves $\lmfdbec{15a7}{15.a4}$, $\lmfdbec{15a8}{15.a7}$, $\lmfdbec{48a4}{48.a5}$, where $c_E=1.$
	\item If $N=6,$ then $2|c_E,$ except for the curve $\lmfdbec{20a2}{20.a3}$, where $c_E=3,$  and also the curves $\lmfdbec{80b2}{80.b3}$, $ \lmfdbec{80b4}{80.b1},$ $\lmfdbec{20a4}{20.a1}$, $\lmfdbec{27a3}{27.a4}$ and infinitely many twists of $\lmfdbec{27a3}{27.a4}$, for which $c_E=1$.
	\item If $n\in\{14,17,19,37,43,67,163\}$, then $2|c_E.$
\end{itemize}
\end{theorem}
The proof of this theorem is given by the proofs of Propositions \ref{prop1}, \ref{10izo}, \ref{8izo}, \ref{6izo} and \ref{ostalo} in Section \ref{iso}.

Now let $E$ be an elliptic curve defined over $K_v$, given by a Weierstrass equation
$$y^2+a_1xy+a_3y=x^3+a_2x^2+a_4x+a_6.$$
with discriminant $\Delta$, invariants $c_4$ and $c_6$, and $j-$invariant $j_E=\frac{c_4^3}{\Delta}.$
It will be important for us to distinguish between different types of reductions at finite primes, especially to know when the reduction is multiplicative. For that, we will often use the following well known result.

\begin{proposition}\label{prop:mult}
	\textup{(see \cite[Proposition VII.5.1.b]{S})} With the above notation, the curve $E$ in its minimal model has multiplicative reduction at $v$ of type $I_k$ if and only if $k:=\ord_{v}(\Delta)>0$ and $\ord_{v}(c_4)=0$.
\end{proposition}
As most Tamagawa numbers that we will consider in this paper are coming from primes of multiplicative reduction, it will be important to also distinguish between split and non-split multiplicative reductions and their Tamagawa numbers. One way to do that is by using the algorithm of Tate \cite[Sections 7,8]{T} which works in any characteristic of $k_v.$ Going through the algorithm with a specific elliptic curve and a prime $p$, we get the reduction type at $p$, its Kodaira symbol and the Tamagawa number $c_p.$ It turns out that in the case of split multiplicative reduction $I_k$ we have $c_v=k$ and in the case of non-split multiplicative reduction $I_k$ we have $c_v=1$ or $c_v=2$, depending on the parity of $k$, as indicated in Table \ref{table1}, where we can find all the Tamagawa numbers associated to different reduction types. For distinguishing reduction types in $char(k_v)\neq 2,3$ one can also use the tables in \cite[Table 15.1]{S} or \cite[Section 6]{T}.
\renewcommand{\arraystretch}{1.2}

\begin{table}[h]
	\begin{tabular}{|c|c|c|}
		\hline
		reduction type at $v$ & Kodaira symbol, $k\geq 1$ & Tamagawa number at $v$ \\ \hline
		good & $I_0$ & 1 \\ 
		split multiplicative  & $I_k$          & $k$                      \\
		non-split multiplicative& $I_{2k}$     & $2$        \\
		non-split multiplicative &  $I_{2k-1}$              &   $1$                    \\
		additive& $II,II^*$              & 1                      \\
		additive& $III,III^*$        & 2                      \\
		additive& $IV,IV^*$         & 1,3                      \\
		additive   & $I_0^*$              & $1,2,4$                      \\
		potentially multiplicative   & $I_{2k}^*$            & $2,4$                     \\
		potentially multiplicative    & $I_{2k-1}^*$          & $2,4$               \\ \hline
	\end{tabular}
	\vspace{0,2cm}
	\caption{types of reduction and their Tamagawa numbers}\label{table1}
\end{table}

The computations in this paper were executed in the computer algebra system
Magma \cite{MAGMA}. The code used in this paper can be found at
\url{https://web.math.pmf.unizg.hr/~atrbovi/magma/magma3.htm}. Many of the proofs in this paper omit the information used in them, such as polynomials of very high degree or with large coefficients, but those can be computed with the given code. For the reader who wants to verify the calculations, we recommend that they go through the proofs and the code simultaneously.

\noindent All of the specific curves will be mentioned using their LMFDB labels, with a clickable link to the corresponding webpage in \cite{LMFDB}.

\section{Tamagawa numbers of elliptic curves with torsion subgroup $\Z/2\Z \oplus \Z/14\Z$}\label{tors}

As already mentioned, Bruin and Najman \cite{BN} proved that every elliptic curve with torsion $\Z/2\Z \oplus \Z/14\Z$ over a cubic field is a base change of an elliptic curve defined over $\Q.$ Filip Najman and the author have examined the reduction types at primes with multiplicative reduction of such elliptic curves defined over $\Q$ in \cite[Prop. 3.1]{NT}. We will examine those primes further, as we want to be able to say more about their Tamagawa numbers.
It was proved in \cite[Prop. 3.2]{NT} that those elliptic curves always have multiplicative reduction of type $I_{14k}$ at the rational prime 2.
In this section we are going to prove that the mentioned multiplicative reduction at 2 always has to be split multiplicative, giving the Tamagawa number $c_2=14k$, as shown in Table \ref{table1}. We are also going to prove that there always exists one more prime $p$, with the exception of the curve \lmfdbec{1922c1}{1922.e2}, at which we have split multiplicative reduction of type $I_{14t}$ and $c_p=14t$, which means that the Tamagawa number of the elliptic curve contains the factor $14^2.$
For the base change of every elliptic curve to a field $K$ over which they have torsion $\Z/2\Z \oplus \Z/14\Z$ it turns out that their Tamagawa number is always divisible by $14^3.$

Bruin and Najman in \cite{BN} also showed that elliptic curves with torsion $\Z/2\Z \oplus \Z/14\Z$ are parameterized with $\mathbb{P}^1(\Q)$, so we can write each such curve as $E_u$, for some $u\in\Q.$ They also provided a model, which was used for obtaining the results of \cite[\S 3]{NT}.
We used a different model here, specifically, the one given by Jeon and Schweizer in \cite[\S 2.4]{JS}, since the one in \cite{NT} was dependant on 2 parameters. It did not impose a problem there, since we did not have the need to work with the coefficients of the curve. Even though Jeon and Schweizer do not state that their family consists of all elliptic curves over cubic fields with torsion $\Z/2\Z \oplus \Z/14\Z$, it turns out that it is the case and the reasoning behind it can be found in the accompanying \href{https://web.math.pmf.unizg.hr/~atrbovi/magma/magma3/2families.txt}{Magma code}. Briefly, we compute the isomorphism between different fields of definition of elliptic curves with torsion  $\Z/2\Z \oplus \Z/14\Z$, those are $F$ and $L$ given in \cite{BN} and \cite[\S 2.4]{JS}, respectively. With that isomorphism we map every curve from the family in \cite{BN} and we see that it is isomorphic to one of the curves from the family in \cite[\S 2.4]{JS}. Since \cite{BN} gives us all of the elliptic curves with needed properties, we see that it suffices to only look at the family from \cite[\S 2.4]{JS}.

Jeon and Schweizer provided two models for $E_u$, one of which is 
$$y^2+xy=x^3+A_2(u)x^2+A_4(u)x+A_6(u),$$
and its short Weierstrass model 
$$y^2=x^3+A(u)x+B(u),$$
where we omit $A_2(u),A_4(u), A_6(u), A(u), B(u)$, since they are very large, but thay can be found in the accompanying \href{https://web.math.pmf.unizg.hr/~atrbovi/magma/magma3.htm}{Magma code} or in \cite[\S 2.4]{JS}.
We will be working with the long Weierstrass model when considering the reduction at the prime 2, but generally we will be using the short Weierstrass model, since it is easier to work with.



In \Cref{prop:mult} we mentioned a way of confirming whether the curve has multiplicative reduction at a finite prime. As already stated, it will be very important to distinguish between split and non-split multiplicative reduction, since the associated Tamagawa numbers are different (see \Cref{table1}). The following lemma will be useful in differentiating between those, and it is taken directly from a step in Tate's algorithm.

\begin{lemma}\label{lema1}
	\textup{(\cite[\S 7. Case 2)]{T})} Let $E$ be an elliptic curve and let $p$ be a prime of multiplicative reduction of type $I_{t}$ for $E$. Let $\ord_p(a_i)>0,$ for $i=3,4,6$, and $\ord_p(b_2)=0.$ If $T^2+a_1T-a_2$ splits over $k_p$, then $E$ has split multiplicative reduction at $p$ and $c_p=t.$
\end{lemma}

As a part of the proof of the following proposition we will show that the reduction at the prime 2 is multiplicative of type $I_{14k}$, which is already proved in \cite[Proposition 3.2]{NT}. We had to include it here again and could not continue from there because of the already mentioned differences in the models we used.

\begin{proposition}\label{prop:p=2}
	Let $E$ be an elliptic curve defined over $\Q$ with torsion subgroup $\Z/2\Z \oplus \Z/14\Z$ over a cubic field. Then the reduction at $2$ is split multiplicative of type $I_{14k}$ and $c_2=14k,$ where $k\in\Z, \: k\geq 1.$
\end{proposition}

\begin{proof}
From the long Weierstrass model of $E_u$ from \cite[\S 2.4]{JS} we get the associated discriminant
$$\Delta(u)=\dfrac{2^{14}(u-1)^{14}(u+1)^{14}f_1(u)}{f_2(u)}$$
and the $c_4-$invariant $c_4(u)$, which can be computed with the accompanying \href{https://web.math.pmf.unizg.hr/~atrbovi/magma/magma3/prop2.2.txt}{Magma code}. The polynomials $f_{i}(u), \: i=1,2,$ are monic polynomials in $\Z[u]$. We will go through all of the possibilities of the prime 2 dividing $u$ and see that the reduction at 2 in all of those cases is split multiplicative and $14\mid c_2.$
\begin{itemize}
	\item[(1)] If $\ord_2(u)>0,$ we compute $\ord_2\left(\frac{\Delta(u)}{2^{14}} \right)=\ord_2(c_4(u))=0$ and from \Cref{prop:mult} we conclude that the reduction at 2 is multiplicative of type $I_{14}.$ We compute $a_1=1$ and $\ord_2(a_2)>0$ and since our model satisfies the conditions of Lemma \ref{lema1}, we get that $c_2=14.$
	

	\item[(2)] If $\ell:=\ord_2(u)<0,$ then we make the substitution $u\mapsto \frac{1}{m}$ so $\ord_2(m)>0,$ and in the new model we get the discriminant
	$$\Delta(m)=\dfrac{2^{14}(m-1)^{14}m^{14}(m+1)^{14}g_1(m)}{g_2(m)}$$
	and the $c_4-$invariant $c_4(m)$, which can be computed with the accompanying \href{https://web.math.pmf.unizg.hr/~atrbovi/magma/magma3/prop2.2.txt}{Magma code}. The polynomials $g_{i}(m), \: i=1,2,$ are monic polynomials in $\Z[m]$. Using the fact that $\ord_2(m)>0,$ we compute $\ord_2\left(\frac{\Delta(m)}{2^{14}m^{14}} \right)=\ord_2(c_4(m))=0$ and as in the previous case, using \Cref{prop:mult} and \Cref{lema1} we get that the reduction at 2 is split multiplicative of type $I_{14(\ell+1)}$ and $c_2=14(\ell+1).$
	\item[(3)] If $\ord_2(u)=0,$ then $\ell:=\ord_2(u-1)>0.$ After the substitution $u-1\mapsto m$ we have $\ell=\ord_2(m)>0$, the discriminant
	$$\Delta(m)=\dfrac{2^{14}m^{14}(m+2)^{14}h_1(m)}{h_2(m)}$$
	and the $c_4-$invariant $c_4(m)$, which can be computed with the accompanying \href{https://web.math.pmf.unizg.hr/~atrbovi/magma/magma3/prop2.2.txt}{Magma code}. The polynomials $h_{i}(m), \: i=1,2,$ are monic polynomials in $\Z[m]$.
	\noindent We can divide both the numerator and the denominator of $\Delta(m)$ by $2^{48}$ and we get that
	$\ord_2(\Delta (m))=14(\ell-1)$ and if we divide the numerator and the denominator of $c_4(m)$ by $2^{24}$ we get that $\ord_2(c_4(m))=0.$ So if $\ell>1,$ by \Cref{prop:mult} we have that the reduction at 2 is multiplicative of type $I_{14(\ell-1)}$. We compute $a_1=1$ and $\ord_2(a_2)>0$ (after dividing both numerator and the denominator by $2^{12}$) and since our model satisfies the conditions of Lemma \ref{lema1}, we get that $c_2=14(\ell-1).$
	
	\noindent Obviously we have to look at the case $\ell=1$ separately. This means that $u=2n+1,$ where $\ord_2(n)=0.$ After the substitution  $u\mapsto 2n+1$ we have the discriminant
	$$\Delta(n)=\dfrac{n^{14}(n+1)^{14}p_1(n)}{p_2(n)}$$
	and the $c_4-$invariant $c_4(n)$, which can be computed with the accompanying \href{https://web.math.pmf.unizg.hr/~atrbovi/magma/magma3/prop2.2.txt}{Magma code}. The polynomials $p_{i}(n), \: i=1,2,$ are monic polynomials in $\Z[n]$. Since $\ord_2(n)=0,$ we have $t:=\ord_2(n+1)>0$ and $\ord_2(p_i(n))=0,$ for each $i$, so $\ord_2(\Delta (n))=14t$ and $\ord_2(c_4(n))=0$. By \Cref{prop:mult} we see that the reduction at 2 is multiplicative of type $I_{14t}$ and similarly as in previous cases, Lemma \ref{lema1} gives that the reduction is split multiplicative with $c_2=14t.$
\end{itemize}
\end{proof}

\begin{example}
	As it was verified in part (1) of the proof of Proposition \ref{prop:p=2}, if $\ord_2(u)>0,$ then the reduction at 2 is multiplicative of type $I_{14}$ with $c_2=14.$ This allows us to generate an infinite family of elliptic curves that have torsion $\Z/2\Z \oplus \Z/14\Z$ over a cubic field and Tamagawa number exactly $14$ at the prime $2$, i.e., $c_2=14.$ Namely, if we put $u=2^k,$ for any $k\in\Z, \: k\geq 1$, in the long Weierstrass model of $E_u$ from \cite[\S 2.4]{JS} we will get an elliptic curve with $c_2=14.$ For example, with $k=1$ (which gives $u=2$) we get a curve whose minimal model is defined by
	$$y^2 + xy = x^3 - 31714388875x + 2132064170125553,$$
	with $c_2=14$ and torsion subgroup $\Z/2\Z \oplus \Z/14\Z$ over the field $\Q(\alpha),$ where $\alpha$ is a root of the polynomial $3x^3-4x^2-27x+4.$
	
	In a similar manner, part (2) of the proof of Proposition \ref{prop:p=2} allows us to generate an infinite family of examples of elliptic curves that have torsion $\Z/2\Z \oplus \Z/14\Z$ over a cubic field and Tamagawa number $c_2=14t,$ where $t>1.$ When $k:=\ord_2(u)<0,$ the long Weierstrass model of $E_u$ from \cite[\S 2.4]{JS} gives us an elliptic curve with $c_2=14(k+1)$ and torsion subgroup $\Z/2\Z \oplus \Z/14\Z$ over a cubic field. Namely, if we specify $u=\frac{1}{2^k}, \: k\geq 1,$ we get a family of elliptic curves with Tamagawa number $c_2=14(k+1)$. For example, with $k=1$ (which gives $u=\frac{1}{2}$) we get a curve whose minimal model is defined by 
	$$y^2 + xy = x^3 - 35365397163613670x +
	2559848051274532647229668,$$ with $c_2=28$ and torsion subgroup $\Z/2\Z \oplus \Z/14\Z$ over the field $\Q(\alpha),$ where $\alpha$ is a root of the polynomial $-6x^3-47x^2+54x+47.$
	
	All of the statements regarding specific elliptic curves in this example can be verified using the accompanying \href{https://web.math.pmf.unizg.hr/~atrbovi/magma/magma3/ex2.3.txt}{Magma code}.	
\end{example}

\begin{corollary}\label{kor}
	Let $E$ be an elliptic curve defined over a cubic field $K$ with torsion subgroup $\Z/2\Z \oplus \Z/14\Z$. Let $\mathfrak{P}$ be a prime of $K$ over $2$. Then the reduction at $\mathfrak{P}$ is split multiplicative of type $I_{14k}$ and $c_\mathfrak{P}=14k,$ where $k\in\Z, \: k\geq 1.$ Furthermore, $14^3 \mid c_E.$
\end{corollary}
\begin{proof}
	Recall that $E$ is an elliptic curve defined over $\Q$. We will denote by $E_K$ the base change of $E$ to $K$. From \cite[Proposition 3.6]{NT} we know that $2$ splits completely in $K$, i.e. $2\mathcal{O}_K=\mathfrak{P}_1\cdot \mathfrak{P}_2\cdot \mathfrak{P}_3.$ This means that the residue field $k_{\mathfrak{P}_i}=\mathcal{O}_K/\mathfrak{P}_i,$ where $\mathfrak{P}_i$ is a prime lying over $2$, $i=1,2,3$, is isomorphic to $k_p$. For each $\mathfrak{P}_i$ we have that $E_K \text{ mod } \mathfrak{P}_i=E \text{ mod } 2$ and hence $c_{\mathfrak{P}_i}=c_2=14k,$ for $i=1,2,3,$ where $k\in\Z, \: k\geq 1.$
\end{proof}

In the following proposition we will deal with primes distinct from 2, for which we have a simpler way of determining split multiplicative reduction than going through Tate's algorithm as we did in Proposition \ref{prop:p=2}.

\begin{lemma}\label{lema2}
	\textup{(\cite[Lemma 2.2]{CCH})} Let $p\neq 2$ be a prime and let $E$ be an elliptic curve defined over $\Q_p$ with multiplicative reduction at $p$. The reduction is split multiplicative if and only if $-c_6$ is a square in $\mathbb{F}_p^{\times}.$
\end{lemma}

\begin{proposition}\label{main}
	Let $E$ be an elliptic curve defined over $\Q$ with torsion subgroup $\Z/2\Z \oplus \Z/14\Z$ over a cubic field. Then there exist at least 2 rational primes with split multiplicative reduction of type $I_{14k},$ where $k\in\Z, \: k\geq 1,$ one of which is always the prime $2$, so $14^2 \mid c_E,$ except for the curve $\lmfdbec{1922c1}{1922.e2}$, where $c_E=c_2=14.$
\end{proposition}

\begin{proof}
In \Cref{prop:p=2} we have already seen that the reduction at 2 is split multiplicative of type $I_{14k}$ and therefore $c_2=14k.$ It remains to prove that there exists one more prime with the same property for each of those curves.

From the short Weierstrass model of $E_u$ from \cite[\S 2.4]{JS} we get the associated discriminant
	$$\Delta(u)=2^{14}(u - 1)^{14}(u + 1)^{14}f(u)$$
	and the $c_4-$invariant $c_4(u)$, which can be computed with the accompanying \href{https://web.math.pmf.unizg.hr/~atrbovi/magma/magma3/prop2.6.txt}{Magma code}. The polynomial $f(u)$ is a monic polynomial in $\Z[u]$.	
	\begin{itemize}
		\item Assume that there exists a prime $p$ such that $k:=\ord_p(u-1)>0.$ Let $\res(q, r)$ denote the resultant of two arbitrary polynomials $q$ and $r$. We compute
		$$\res\left(u-1, \frac{\Delta(u)}{(u-1)^{14}}\right)=2^{82}, $$
		$$\res\left(u-1, c_4(u)\right)=2^{32}. $$
		For $p\neq 2$ this means that $p^{14k}\mid\Delta(u)$ and $p\nmid c_4(u)$ and from \Cref{prop:mult} we find that the reduction of $E$ at $p$ is multiplicative of type $I_{14k}.$ We want to see that $c_p=14k,$ i.e., that the reduction at $p$ is split multiplicative. According to \Cref{lema2}, it will suffice to check the value of $-c_6$ modulo $p$. Having in mind that $u\equiv 1 \pmod p$, we compute that $-c_6\equiv 2^{48} \pmod p,$ which is a square mod $p$.
		
		\item Assume now that there exists a prime $p$ such that $k:=\ord _p(u-1)<0.$ We put $m:=\frac{1}{u-1}$ so $\ord_p(m)=k>0$ and we get an elliptic curve with the discriminant
		$$\Delta(m)=\frac{1}{2^{14}}m^{14}(m + 1/2)^{14}g(m)$$
		and the $c_4-$invariant $c_4(m)$, which can be computed with the accompanying \href{https://web.math.pmf.unizg.hr/~atrbovi/magma/magma3/prop2.6.txt}{Magma code}. The polynomial $g(m), \: i=1,2,$ is a monic polynomial in $\Z[u]$.
		
		\noindent We compute
		$$\res\left(m, \frac{\Delta(m)}{m^{14}}\right)=2^{-82}, $$
		$$\res\left(m, c_4(m)\right)=2^{-32}. $$
		For $p\neq 2$ this means that $p^{14k}\mid\Delta(m)$ and $p\nmid c_4(m)$ and from \Cref{prop:mult} we find that the reduction of $E$ at $p$ is multiplicative of type $I_{14k}.$
		Having in mind that $m\equiv 0 \pmod p$, we get that $-c_6\equiv 2^{-48} \pmod p,$ which is a square mod $p$, so by \Cref{lema2} we have $c_p=14k.$
	\end{itemize}
So far we have proved that if we have a prime $p\neq 2$ and $k:=\ord_p(u-1)\neq 0$, then we have split multiplicative reduction at $p$ with $c_p=14|k|.$ We have several possibilities when $\ord_p(u-1)= 0$ and those are $u-1=0$ or $u-1=\pm 2^k,\: k\in\Z.$

\noindent When $u-1=\pm 2^k,\: k\neq 0,1,$ then $\ord_p(u+1)>0,$ for some prime $p\neq 2.$ In the cases $u-1=\pm 2^k,\: k=0,1$, or $u-1=0$ we get that $u\in \{0,\pm 1, 3\}.$ For $u=\pm 1$ we get a singular curve and for $u\in \{0, 3\}$ we get the same curve, $\lmfdbec{1922c1}{1922.e2}$, with $c_E=c_2=14.$

\noindent Therefore, if we have a curve distinct from $\lmfdbec{1922c1}{1922.e2}$, it certainly has a prime $p$ such that $\ord_p(u-1)\neq 0$ or $\ord_p(u+1)> 0.$ It remains to see what happens in the case $\ord_p(u+1)> 0.$

\begin{itemize}
	\item Assume that there exists a prime $p$ such that $k:=\ord_p(u+1)>0.$ We compute
	$$\res\left(u+1, \frac{\Delta(u)}{(u+1)^{14}}\right)=2^{82}, $$
	$$\res\left(u+1, c_4(u)\right)=2^{32}. $$
	For $p\neq 2$ this means that $p^{14k}\mid\Delta(u)$ and $p\nmid c_4(u)$ and from \Cref{prop:mult} we find that the reduction of $E$ at $p$ is multiplicative of type $I_{14k}.$ Having in mind that $u\equiv -1 \pmod p$, we get that $-c_6\equiv 2^{48} \pmod p,$ which is a square mod $p$, so by \Cref{lema2} we have $c_p=14k.$
\end{itemize}

\end{proof}

\begin{remark}
 In Proposition \ref{main} we proved that $c_E=14$ is the least possible value of the Tamagawa number of an elliptic curve defined over $\Q$ with torsion subgroup $\Z/2\Z\oplus \Z/14\Z$ over some cubic field. This is true only for elliptic curve $E=\lmfdbec{1922c1}{1922.e2}$. Consequently, for the same curve we get the least possible value amongst those curves of the ratio $c_E/\#E(\Q)_{tors}$ which appears as a factor in the leading term of the $L-$function of $E/\Q$ in the conjecture of Birch and Swinnerton-Dyer, in this case it is $c_E/\#E(\Q)_{tors}=\frac{14}{1}=14,$ since the curves defined over $\Q$ with torsion subgroup $\Z/2\Z\oplus \Z/14\Z$ over some cubic field have trivial torsion over $\Q.$
 
 Similarly, using Corollary \ref{kor} we see that the value of the Tamagawa number of an elliptic curve defined over a cubic field with torsion subgroup $\Z/2\Z\oplus \Z/14\Z$ is always divisible by $14^3.$ The value $c_E=14^3$ is actually a possible value, and it is achieved for the curve $E=\lmfdbec{1922c1}{1922.e2}$, which has the mentioned torsion over the cubic field $\Q(\alpha),$ where $\alpha$ is a root of the polynomial $x^3 + 2x^2 - 9x - 2$. This gives the smallest possibe ratio of $c_E/\#E(K)_{tors}=98$ for all such curves.
\end{remark}

\section{Tamagawa numbers of elliptic curves with prescribed isogeny}\label{iso}

In \cite[Table 3]{ALR} we can find the $j-$invariants of elliptic curves parameterized by points on modular curves $X_0(n)$ defined over $\Q,$ for $X_0(n)$ of genus $0$, and in \cite[Table 4]{ALR} there are $j-$invariants of elliptic curves parameterized by points on modular curves $X_0(n)$ defined over $\Q,$ with genus of $X_0(n)$ larger than $0$. In this section we will examine the properties of Tamagawa numbers of elliptic curves defined over $\Q$ with an $n-$isogeny, i.e., the properties of Tamagawa numbers of elliptic curves obtained from the mentioned $j-$invariants.

In Section \ref{tors} we worked with a specific model for the curve $X_1(2,14)$. However, the points on $X_0(n)$ give us $j-$invariants of curves with an $n-$isogeny, which give us elliptic curves up to a twist, so now, as opposed to the situation in Section \ref{tors}, we also have to take into consideration the twists of the curves we get from those $j-$invariants. Therefore, we will be interested in how the reduction types at primes $p\in\Q$ change under the twisting of the curve.

Let $E$ be an elliptic curve, which will always be defined over $\Q$ in this section. Denote by $E^d$ its quadratic twist by $d$, where $d$ is a squarefree integer. When $p\neq 2$, the reduction type change is quite straightforward, and is presented in \Cref{tab2}. In essence, if $p\nmid d,$ the reduction type does not change, and when $p\mid d,$ reduction types change as indicated in the third column.

\begin{table}[h]
	\begin{tabular}{|c|c|c|}
		\hline
		\begin{tabular}[c]{@{}c@{}}reduction type \\ of $E$ at $p$\end{tabular} & \begin{tabular}[c]{@{}c@{}}reduction type\\  of $E^d$ at $p\nmid d$\end{tabular}  &  \begin{tabular}[c]{@{}c@{}}reduction type\\  of $E^d$ at $p\mid d$\end{tabular}\\ \hline
		$I_0$                        &   $I_0$   &  $\:\:\:I_0^*\:$                   \\
		$\:I_m$                      &  $\:I_m$    &  $\:\:\:\:I_m^*\:$ 	                      \\
		$II$                         &     $II$  &      	$\:\:\:\:\:IV^*$                \\
		$\:\:III$                    &  $\:\:III$      &  	$\:\:\:\:\:\:III^*$                      \\
		$\:IV$                       &  $\:IV$         &   	$\:\:\:II^*$                \\
		$I_0^*\:$                    & $I_0^*\:$   &     $I_0$                         \\
		$\:I_m^*\:$                  & $\:I_m^*\:$  & 	$\:I_m$                            \\
		$\:\:IV^*$                   & 	$\:\:IV^*$   &    $II$                        \\
		$\:\:III^*$                      &    	$\:\:III^*$    &     $\:\:III$                     \\
		$\:\:II^*\:\:$                   &	$\:\:II^*\:\:$  &  $\:IV$                                   \\  \hline
	\end{tabular}
	\vspace{0,2cm}
	\caption{change of reduction types at $p\neq 2$ under twisting \cite[Prop.1]{SC}}\label{tab2}
\end{table}

When $p=2,$ the situation gets more complicated. As most of the relevant Tamagawa numbers we will have in the following proofs come from primes of multiplicative reduction, we give a lemma that will be especially useful for dealing with quadratic twists of a large family of elliptic curves with multiplicative reduction at $p=2.$

\begin{lemma}\label{lemma1}
	\textup{(\cite[Thm.A.5]{D2}, \cite[Thm.2.8]{DL2})} Let $E$ be an elliptic curve with multiplicative reduction of type $I_n$ at $p=2.$ Denote by $E^d$ the twist of $E$ by $d,$ where $d$ is a squarefree integer.
	\begin{itemize}
		\item[(a)] If $d\equiv 2,3\pmod{4},$ then the reduction of $E^d$ at $p$ is of type $I_n^*.$
		\item[(b)] If $d\equiv 1\pmod{4},$ then the reduction of $E^d$ at $p$ is of type $I_n.$
	\end{itemize}
\end{lemma}

For other types of reduction, some results can also be found in \cite[Section 2]{SC}. Since we will deal here with only finitely many explicitly known elliptic curves with non-multiplicative reduction at $p=2,$ for those curves we can simply check all of the possibilities for reduction type at $p=2$ of quadratic twists, since $\Q_2^\times/\left( \Q_2^\times\right)^2=\langle -1,2,5 \rangle$.


\begin{proposition}\label{prop1}
	Let $E$ be an elliptic curve over $\Q$ with an $18-$isogeny. Then $4|c_E,$ except for the curves $\lmfdbec{14a3}{14.a2}, \lmfdbec{14a4}{14.a5}, \lmfdbec{14a5}{14.a1}, \lmfdbec{14a6}{14.a4}$, where $ c_E=2$.
\end{proposition}

\begin{proof}
	From \cite[Table 3]{ALR} we take the parameterization of the $j-$invariants of the curves that are non-cuspidal points on $X_0(18),$
	$$j(h)=\dfrac{(h^3-2)^3(h^9-6h^6-12h^3-8)^3}{h^9(h^3-8)(h^3+1)^2}, h\in\Q.$$
	
	From it we can acquire the discriminant
	$$\Delta(h)=(h-2)h^9(h+1)^2(h^2 - h + 1)^2(h^2 + 2h + 4)f(h)$$
	and the $c_4-$invariant $c_4(h)$ of the minimal model up to a twist, which can be computed with the accompanying \href{https://web.math.pmf.unizg.hr/~atrbovi/magma/magma3/prop3.2.txt}{Magma code}. The polynomial $f(h)$ is a monic polynomial in $\Z[h].$
	
	Assume that there exists a prime $p$ such that $k:=\ord _p(h+1)>0.$ We compute
	$$\res\left(h+1, \frac{\Delta(h)}{(h+1)^2}\right)=3^{34}, $$
	$$\res\left(h+1, c_4(h)\right)=3^{12}. $$
	If $p\neq 3,$ this means that $p^{2k}\mid\Delta(h)$ and $p\nmid c_4(h)$, so from \Cref{prop:mult} we find that the reduction of $E$ at $p$ is multiplicative of type $I_{2k},$ and therefore $c_p$ is even (see \Cref{table1}).
	If there exists a second prime $p'$ distinct from $p$ and 3 with $k':=\ord _{p'}(h+1)>0,$ then we have another prime with  multiplicative reduction of type $I_{2k'}$ and therefore with even $c_{p'}.$

	Assume now that there exists a prime $p$ such that $k:=\ord _p(h+1)<0.$ We put $m:=\frac{1}{h+1}$ and after the substitution $x\mapsto x\cdot 3^{-6}, y\mapsto y\cdot 3^{-9}$ we get an elliptic curve with the discriminant	
	$$ \Delta(m)=(m - 1)^9(3m-1)m^{18}(3m^2 -3 m + 1)^2(3m^2 + 1)g(m)$$
	and the $c_4-$invariant $c_4(h)$ up to a twist, which can be computed with the accompanying \href{https://web.math.pmf.unizg.hr/~atrbovi/magma/magma3/prop3.2.txt}{Magma code}. The polynomial $g(m)$ is a monic polynomial in $\Z[m].$
	
	\noindent We compute
	$$\res\left(m, \frac{\Delta(m)}{m^{18}}\right)=1, $$
	$$\res\left(m, c_4(m)\right)=1. $$

	We see from \Cref{prop:mult} that the reduction at $p$ is multiplicative of type $I_{18k},$ with even $c_p$ (see \Cref{table1}). If we have another prime $p'\neq p$ with $k':=\ord _{p'}(h+1)<0,$ then the reduction at $p'$ is also multiplicative of type $I_{18k'}$ with even $c_{p'}.$
	
	Assume now that there exists at most one prime $p$ such that $\ord_p(h+1)\neq 0.$ That means that either $h+1=\pm p^k,$ where $k\in \Z, \: k\geq 0,$ or $m=\frac{1}{h+1}=\pm p^k,$ where $k\in \Z, \: k> 0.$ We consider the following cases:
	
	\begin{itemize}
		\item[(1)]  If $h+1=\pm p^k, p\neq 3,$ we have $\ord _{p}(h^2-h+1)=0,$ since $\res(h+1, h^2-h+1)=3$, where $h^2-h+1$ is one of the factors in the discriminant. Then there exists $p'\neq p,3$ such that $\ord_{p'}(h^2-h+1)>0$, unless we have $h^2-h+1\in\{\pm 1, \pm 3\},$ i.e., $h\in\{0,\pm 1, 2\}.$ For $h=1$ we get a twist of the curve \lmfdbec{14a4}{14.a5} which has $c_E=2,$ while for $h=0, -1,2$ we do not get an elliptic curve (look at the j-invariant).

		\item[(2)] Suppose $h+1=\pm 3^k.$ If $k=0$ we have $h+1=\pm 1,$ i.e., $h\in \{0,-2\}.$ We already know that $h$ cannot be 0, but for $h=-2$ we get a twist of the curve \lmfdbec{14a6}{14.a4}, for which we have $c_E=2.$ When $k=1$ we have $h+1=\pm 3,$ i.e., $h\in\{-4,2\}.$ For $h=-4$ we get a twist of \lmfdbec{14a5}{14.a1}, with $c_E=2,$ and $h=2$ cannot happen. Assume now that $h+1=\pm 3^k, k>1.$ Counting the multiplicities of 3 in $\Delta(h)$ and $c_4(h)$ we get that the factor $3^{2-2k}$ appears in the j-invariant. Furthermore, if we write $\pm 3^k-1$ instead of $h$ in the equation for $E$ and make the substitution $x\mapsto x\cdot 3^{6}, y\mapsto y\cdot 3^{9},$ we get a model where $\ord_3(c_4)=0,$ and it follows from \Cref{prop:mult} that for $k>1$ we have multiplicative reduction $I_{2k-2}$ at 3, with $c_3$ being even (see \Cref{table1}). Note that in any case we also have a prime $p\neq 3$ dividing $h^2-h+1$ in $\Delta(h)$ with multiplicative reduction $I_{2n},$ which makes $c_E$ divisible by 4.
		
		\item[(3)] 	If $m=\frac{1}{h+1}=\pm p^k,$ for some prime $p,$ and clearly $\ord _{p}(3m^2-3m+1)=0,$ since $\res(m, 3m^2-3m+1)=1$. Then there exists $p'\neq p$ such that $\ord_{p'}(3m^2-3m+1)>0.$ Otherwise, we have $3m^2-3m+1\in\{\pm 1\},$ i.e., $m\in\{0,1\}$ which only makes sense for $h=0$ but, as we noted earlier, $h$ cannot be 0.
		
	\end{itemize}

	The only thing left to consider is when we have only 2 primes with $\ord_p(h+1)\neq 0,$ one of which is 3 and divides the numerator; in other words the cases $h+1=\pm 3^kp^l$ and $h+1=\pm \frac{3^k}{p^l}, p\neq 3, k, l>0.$ From the reasoning in (2) above, it is clear that if $k>1,$ we have multiplicative reduction at 3 and from the part of the proof where we had $\ord_p(h+1)<0$ we see that the reduction is multiplicative at $p$ as well, which gives us $c_E$ that is divisible by 4. For $k=1,$ we have $h+1=\pm 3p^l$ or $h+1=\pm \frac{3}{p^l}.$ 
	
	\begin{itemize}
		\item If $h+1=\pm 3p^l,$ we have another prime $p'\neq p,3$ dividing $h^2-h+1$ in the discriminant (similarly as in (1)) with multiplicative reduction.
		
		\item 	If $h+1=\frac{1}{m}=\pm \frac{3}{p^l},$ we also have another prime $p'\neq p$ dividing the numerator of $3m^2-3m+1$ in the discriminant (as in (3)) with multiplicative reduction, except possibly when $3m^2-3m+1=\frac{1}{a^n}, a\in\mathbb{Z}, n>0$ (this situation couldn't have happened in (3), because we had $m\in\mathbb{Z}$). By putting $ \pm\frac{p^l}{3} $ instead of $m$, we get $$ \pm\dfrac{p^{2l}}{3}\mp p^l+1=\dfrac{1}{a^n}, $$ which only has solutions for $a=3, n=1, p=2, l=1,$ i.e., if $h\in\left\{-\frac{5}{2},\frac{1}{2}\right\}.$ For $h=\frac{1}{2}$ we get a twist of the elliptic curve \lmfdbec{14a3}{14.a2}, with $c_E=2,$ and for $h=-\frac{5}{2}$ we get a curve that already has 2 primes of reduction type $I_{2k}$, namely 2 and 13.
	\end{itemize}

	To conclude the proof of this proposition, it remains to see how these reduction types and Tamagawa numbers would change under the twisting of the curves. All even Tamagawa numbers mentioned in the proof above come from multiplicative reductions $I_{2n}$ at primes $p,$ so by using Table \ref{table1}, Table 2 and Lemma \ref{lemma1}, we conclude that all reduction types of twists at $p$ are either $I_{2n}$ or $I_{2n}^*,$ so the Tamagawa numbers stay even.
	
	As for the curves \lmfdbec{14a3}{14.a2}, \lmfdbec{14a4}{14.a5}, \lmfdbec{14a5}{14.a1} and \lmfdbec{14a6}{14.a4}, they have $c_E=c_2=2.$ By using the fact that $\Q_p^\times/\left( \Q_p^\times\right)^2=\langle -1,2,5 \rangle$, we explicitly compute all possible reduction types of quadratic twists at $p=2$ and conclude that for every twist of those curves $4\mid c_E.$

\end{proof}

\begin{proposition}\label{10izo}
	Let $E$ be an elliptic curve over $\Q$ with a $10-$isogeny. Then $2|c_E.$
\end{proposition}

\begin{proof}
	From \cite[Table 3]{ALR} we take the parameterization of the $j-$invariants of the curves that are non-cuspidal points on $X_0(10),$	
	$$j(h)=\dfrac{(h^6-4h^5+16h+16)^3}{(h+1)^2(h-4)h^5}, h\in\Q.$$
	
	From it we can acquire the discriminant
	$$\Delta(h)=(h - 4)h^5(h + 1)^2(h^2 - 2h - 4)^6(h^2 - 2h + 2)^6f(h)$$
	and the $c_4-$invariant $c_4(h)$ up to a twist, which can be computed with the accompanying \href{https://web.math.pmf.unizg.hr/~atrbovi/magma/magma3/prop3.3.txt}{Magma code}. The polynomial $f(h)$ is a monic polynomial in $\Z[h].$

	Assume that there exists a prime $p$ such that $k:=\ord _p(h+1)>0.$ We compute
	$$\res\left(h+1, \frac{\Delta(h)}{(h+1)^2}\right)=5^{22}, $$
	$$\res\left(h+1, c_4(h)\right)=5^{8}. $$
	If $p\neq 5,$ this means that $p^{2k}\mid\Delta(h)$ and $p\nmid c_4(h),$ and we find from \Cref{prop:mult} that the reduction of $E$ at $p$ is multiplicative of type $I_{2k},$ and therefore $c_p$ is even (see \Cref{table1}).
	
	For the case $h+1=\pm 5^k,$ after the change of variables $x\mapsto x\cdot 5^{4}, y\mapsto y\cdot 5^{6},$ counting the multiplicities of 5 in $\Delta(h)$ and $c_4(h)$ we get that the factor $5^{2-2k}$ appears in the $j-$invariant, with $5\nmid c_4(h).$ Therefore, when $k>1,$ by \Cref{prop:mult} we have multiplicative reduction at 5 of type $I_{2k-2}$ with even $c_p$ (see \Cref{table1}). For $k\in\{0,1\}$ we have $h\in\{-6,-2,0,4\}.$ For the values $h\in\{0,4\}$ we do not have an elliptic curve, and for the values $h\in\{-6,-2\}$ we get twists of curves \lmfdbec{768d3}{768.h1} and \lmfdbec{768d1}{768.h3}, which have $c_E=2,$ both with bad prime 2 with reduction type $III,$ so $c_2=2.$

	Assume now that there exists a prime $p$ such that $k:=\ord_p(h+1)<0.$ We put $m:=\frac{1}{h+1}$ and after the substitution $x\mapsto x\cdot 5^{-4}, y\mapsto y\cdot 5^{-6}$ we get an elliptic curve with the discriminant	
	$$\Delta(m) =(m - 1)^5(5m - 1)m^{10}(5m^2 - 4m + 1)^6(5m^2 - 2m + 1)^9g(m)$$	
 and the $c_4-$invariant $c_4(m)$ up to a twist, which can be computed with the accompanying \href{https://web.math.pmf.unizg.hr/~atrbovi/magma/magma3/prop3.3.txt}{Magma code}. The polynomial $g(m)$ is a monic polynomial in $\Z[m].$
	
	\noindent We compute
	$$\res\left(m, \frac{\Delta(m)}{m^{10}}\right)=1, $$
	$$\res\left(m, c_4(m)\right)=1. $$
	We see by \Cref{prop:mult} that the reduction at $p$ is multiplicative of type $I_{10k},$ with even $c_p$ (see \Cref{table1}). 
	
	To conclude the proof of this proposition, it remains to see how these reduction types and Tamagawa numbers would change under the twisting of the curves. All even Tamagawa numbers mentioned in the proof above come from multiplicative reductions $I_{2n}$ at primes $p,$ so by using Table \ref{table1}, Table 2 and Lemma \ref{lemma1}, we conclude that all reduction types of twists at $p$ are either $I_{2n}$ or $I_{2n}^*,$ so the Tamagawa numbers stay even.
	
	As for the curves \lmfdbec{768d3}{768.h1} and \lmfdbec{768d1}{768.h3}, they have $c_E=c_2=2.$ By using the fact that $\Q_p^\times/\left( \Q_p^\times\right)^2=\langle -1,2,5 \rangle$, we explicitly compute all possible reduction types of quadratic twists at $p=2$ and conclude that for every twist of those curves $2\mid c_E.$ 
	
\end{proof}

\begin{proposition}\label{8izo}
	Let $E$ be an elliptic curve over $\Q$ with an $8-$isogeny. Then $2|c_E,$ except for the curves $\lmfdbec{15a7}{15.a4}$, $\lmfdbec{15a8}{15.a7}$, $\lmfdbec{48a4}{48.a5}$, where $c_E=1.$
\end{proposition}

\begin{proof}
	From \cite[Table 3]{ALR} we take the parameterization of the $j-$invariants of the curves that are non-cuspidal points on $X_0(8),$	
	$$j(h)=\dfrac{(h^4-16h^2+16)^3}{(h^2-16)h^2}, h\in\Q.$$
	
	From it we can acquire the discriminant
	$$\Delta(h)=(h - 4)h^2(h + 4)f(h)$$
and the $c_4-$invariant $c_4(h)$ up to a twist, which can be computed with the accompanying \href{https://web.math.pmf.unizg.hr/~atrbovi/magma/magma3/prop3.4.txt}{Magma code}. The polynomial $f(h)$ is a monic polynomial in $\Z[h].$

	Assume that there exists a prime $p$ such that $k:=\ord_p(h)>0.$ We compute
	$$\res\left(h, \frac{\Delta(h)}{h^2}\right)=2^{64}, $$
	$$\res\left(h, c_4(h)\right)=2^{24}. $$
	If $p\neq 2,$ then this means that $p^{2k}\mid\Delta(h)$ and $p\nmid c_4(h),$ and from \Cref{prop:mult} we find that the reduction of $E$ at $p$ is multiplicative of type $I_{2k},$ and therefore $c_p$ is even (see \Cref{table1}).
	
	For the case $h=\pm 2^k,$ after the change of variables $x\mapsto x\cdot 2^{12}, y\mapsto y\cdot 2^{18},$ counting the multiplicities of 2 in $\Delta(h)$ and $c_4(h)$ we get that the factor $2^{8-2k}$ appears in the $j-$invariant, with $2\nmid c_4(h).$ Therefore, when $k>4,$ we have multiplicative reduction at 2 of type $I_{2k-8}$ with even $c_p,$ by \Cref{prop:mult} and \Cref{table1}. For $k=0$ we have $h=\pm 1$ and for the both values we get a twist of the curve \lmfdbec{15a8}{15.a7} with $c_E=1.$ For $k=1$ we have $h=\pm 2,$ i.e., a curve \lmfdbec{48a4}{48.a5} up to a twist, with $c_E=1.$ When $k=2$ we do not get a curve and, for $k=3$ and $h=\pm 8$ we have a twist of \lmfdbec{24a3}{24.a2}, where $c_E=2,$ and finally for $k=4$ and $h=\pm 16$ we have a twist of \lmfdbec{15a7}{15.a4}, where $c_E=1.$

	Assume now that there exists a prime $p$ such that $k:=\ord _p(h)<0.$ We put $m:=\frac{1}{h}$ and after the substitution $x\mapsto x\cdot 2^{-12}, y\mapsto y\cdot 2^{-18}$ we get an elliptic curve with the discriminant	
	\begin{equation*}
	\Delta(m) = -(4m - 1)m^8(4m + 1)g(m)
	\end{equation*}
and the $c_4-$invariant $c_4(m)$ up to a twist, which can be computed with the accompanying \href{https://web.math.pmf.unizg.hr/~atrbovi/magma/magma3/prop3.4.txt}{Magma code}. The polynomial $g(m)$ is a monic polynomial in $\Z[m].$

\noindent We compute
	$$\res\left(m, \frac{\Delta(m)}{m^{8}}\right)=1, $$
	$$\res\left(m, c_4(m)\right)=1. $$
	We see from \Cref{prop:mult} and \Cref{table1} that the reduction at $p$ is multiplicative of type $I_{8k},$ with even $c_p.$ 
	
	To conclude the proof of this proposition, it remains to see how these reduction types and Tamagawa numbers would change under the twisting of the curves. All even Tamagawa numbers mentioned in the proof above come from multiplicative reductions $I_{2n}$ at primes $p,$ so by using Table \ref{table1}, Table 2 and Lemma \ref{lemma1}, we conclude that all reduction types of twists at $p$ are either $I_{2n}$ or $I_{2n}^*,$ so the Tamagawa numbers stay even.
	
	As for the curves \lmfdbec{15a7}{15.a4}, \lmfdbec{15a8}{15.a7} and \lmfdbec{48a4}{48.a5}, they have $c_E=1.$ By using the fact that $\Q_p^\times/\left( \Q_p^\times\right)^2=\langle -1,2,5 \rangle$, we explicitly compute all possible reduction types of quadratic twists at $p=2$ and conclude that for every twist of those curves $2\mid c_E.$ Lastly, for every twist of the curve \lmfdbec{24a3}{24.a2} we have $2\mid c_E.$ 
	
\end{proof}

\begin{proposition}\label{6izo}
	Let $E$ be an elliptic curve over $\Q$ with a $6-$isogeny. Then $2|c_E,$ except for the curve $\lmfdbec{20a2}{20.a3}$, where $c_E=3,$  and also the curves $\lmfdbec{80b2}{80.b3}$, $ \lmfdbec{80b4}{80.b1},$ $\lmfdbec{20a4}{20.a1}$, $\lmfdbec{27a3}{27.a4}$ and infinitely many twists of $\lmfdbec{27a3}{27.a4}$, for which $c_E=1$.
\end{proposition}

\begin{proof}
	From \cite[Table 3]{ALR} we take the parameterization of the $j-$invariants of the curves that are non-cuspidal points on $X_0(6),$	
	$$j(h)=\dfrac{(h+6)^3 (h^3+18h^2+84h+24)^3}{h(h+8)^3 (h+9)^2}, h\in\Q.$$
	
	From it we can acquire the discriminant
	$$\Delta(h)=h(h + 6)^6(h + 8)^3(h + 9)^2f(h),$$and the $c_4-$invariant $c_4(h)$ up to a twist, which can be computed with the accompanying \href{https://web.math.pmf.unizg.hr/~atrbovi/magma/magma3/prop3.5.txt}{Magma code}. The polynomial $f(h)$ is a monic polynomial in $\Z[h].$

	Assume that there exists a prime $p$ such that $k:=\ord_p(h+9)>0.$ We compute
	$$\res\left(h+9, \frac{\Delta(h)}{(h+9)^2}\right)=3^{32}, $$
	$$\res\left(h+9, c_4(h)\right)=3^{12}. $$
	If $p\neq 3,$ this means that $p^{2k}\mid\Delta(h)$ and $p\nmid c_4(h),$ and from \Cref{prop:mult} we find that the reduction of $E$ at $p$ is multiplicative of type $I_{2k},$ and therefore $c_p$ is even (see \Cref{table1}).
	
	For the case $h+9=\pm 3^k,$ after the change of variables $x\mapsto x\cdot 3^{6}, y\mapsto y\cdot 3^{9},$ counting the multiplicities of 3 in $\Delta(h)$ and $c_4(h)$ we get that the factor $3^{4-2k}$ appears in the $j-$invariant, with $3\nmid c_4(h).$ Therefore, when $k>2,$ we have multiplicative reduction at 3 of type $I_{2k-4}$ with even $c_p,$ by \Cref{prop:mult} and \Cref{table1}. For $k=0$ we have $h\in\{-10,-8\}.$ With $h=-10$ we have a twist of the elliptic curve \lmfdbec{20a2}{20.a3} which has $c_E=3,$ coming from the reduction at 2 of type $IV,$ and $h=-8$ does not give us an elliptic curve. For $k=1$ we have $h\in\{-12,-6\}.$ For $h=-12$ we get the curve \lmfdbec{36a2}{36.a2} with $c_E=6$ and for $h=-6$ we have \lmfdbec{27a3}{27.a4}, where $c_E=1.$ Lastly, if $k=2,$ then $h\in\{-18,0\}.$ For $h=0$ we do not get an elliptic curve, but for $h=-18$ we get a twist of the curve \lmfdbec{80b4}{80.b1}, with $c_E=1.$

	Assume now that there exists a prime $p$ such that $k:=\ord _p(h+9)<0.$ We put $m:=\frac{1}{h+9}$ and after the substitution $x\mapsto x\cdot 3^{-6}, y\mapsto y\cdot 3^{-9}$ we get an elliptic curve with the discriminant
	\begin{equation*}
	\Delta(m) = (m - 1)^3(3m - 1)^6(9m - 1)m^6g(m)
	\end{equation*}
	and the $c_4-$invariant $c_4(m)$ up to a twist, which can be computed with the accompanying \href{https://web.math.pmf.unizg.hr/~atrbovi/magma/magma3/prop3.5.txt}{Magma code}. The polynomial $g(m)$ is a monic polynomial in $\Z[m].$
	
	\noindent We compute
	$$\res\left(m, \frac{\Delta(m)}{m^{6}}\right)=1, $$
	$$\res\left(m, c_4(m)\right)=1. $$
	We see that the reduction at $p$ is multiplicative of type $I_{6k},$ with even $c_p,$ by \Cref{prop:mult} and \Cref{table1}.

To conclude the proof of this proposition, it remains to see how these reduction types and Tamagawa numbers would change under the twisting of the curves. All even Tamagawa numbers mentioned in the proof above come from multiplicative reductions $I_{2n}$ at primes $p,$ so by using Table \ref{table1}, Table 2 and Lemma \ref{lemma1}, we conclude that all reduction types of twists at $p$ are either $I_{2n}$ or $I_{2n}^*,$ so the Tamagawa numbers stay even.

The curve \lmfdbec{36a2}{36.a2} already has reduction type $III$ at $3$, which stays the same under twisting or changes to $III^*$, as stated in Table \ref{tab2}. In any case we get $c_3=2$ (see Table \ref{table1}). 

As for the curve \lmfdbec{20a2}{20.a3}, we examine all twists by $d$ such that there exists a prime $p\mid d$ such that \lmfdbec{20a2}{20.a3} has good reduction at $p$, i.e., $p\neq 2,5.$ At such $p$ we have reduction type $I_0$, and after twisting by $p$ we get that the reduction type of the twist at $p$ is $I_0^*$, as stated in Table \ref{tab2}. By Tate's algorithm \cite[\S 8. Case 6)]{T} we get that $c_p=1+\text{number of roots of $P(T)$ in $k_p$}$, where $P(T)=T^3+T^2-T=T(T^2+T-1).$ Polynomial $T^2+T-1$ has roots modulo $p$ if and only if $5$ is a quadratic residue modulo $p.$ Therefore, we get $$c_p=\begin{cases}
2, \text{ if $p\equiv 2,3\pmod 5$}\\4, \text{ if $p\equiv 1,4\pmod 5$}
\end{cases}.$$ It remains to see what happens with the twists by $d$, where $d$ has no divisors of good reduction for the curve \lmfdbec{20a2}{20.a3}. Since the only primes of bad reduction are $2$ and $5$, by explicitly computing all possible twists we get that \lmfdbec{20a2}{20.a3}, where $c_E=3$ and \lmfdbec{80b2}{80.b3}, where $c_E=1,$ are the only possible twists for which $2\nmid c_E.$


We use the same approach with the curve \lmfdbec{80b4}{80.b1}, which has $c_E=1.$ We examine all twists by $d$ such that there exists a prime $p\mid d$ such that \lmfdbec{80b4}{80.b1} has good reduction at $p$, i.e., $p\neq 2,5.$ At such $p$ we have reduction type $I_0$, and after twisting by $p$ we get that the reduction type of the twist at $p$ is $I_0^*$, as stated in Table \ref{tab2}. By Tate's algorithm \cite[\S 8. Case 6)]{T} we get that $c_p=1+\text{number of roots of $P(T)$ in $k_p$}$, where $P(T)=T^3 - T^2 - 41T + 116=(T-4)(T^2+3T-29).$ Polynomial $T^2+3T-29$ has roots modulo $p$ if and only if $5$ is a quadratic residue modulo $p.$ Therefore, we get $$c_p=\begin{cases}
2, \text{ if $p\equiv 2,3\pmod 5$}\\4, \text{ if $p\equiv 1,4\pmod 5$}
\end{cases}.$$ It remains to see what happens with the twists by $d$, where $d$ has no divisors of good reduction for the curve \lmfdbec{80b4}{80.b1}. Since the only primes of bad reduction are $2$ and $5$, by explicitly computing all possible twists we get that \lmfdbec{80b4}{80.b1} and \lmfdbec{20a4}{20.a1}, where $c_E=1,$ are the only possible twists for which $2\nmid c_E.$

For the curve \lmfdbec{27a3}{27.a4}, which has $c_E=1$, the situation is more complicated, and we will prove that the curve has infinitely many twists with Tamagawa number 1.

For \lmfdbec{27a3}{27.a4} we have that $c_{E}=c_3=1$, where $3$ is the only prime of bad reduction. Similarly as for the curves \lmfdbec{20a2}{20.a3} and \lmfdbec{80b4}{80.b1}, we are interested in what happens with the Tamagawa numbers of twists $E^d$ when $p\mid d$, for some prime $p$ of good reduction for $E$, in this case $p\neq 3$. At such $p$ we have reduction type $I_0$, and after twisting by $d$ we get that the reduction type of the twist at $p$ is $I_0^*$, as stated in Table \ref{tab2}. By Tate's algorithm \cite[\S 8. Case 6)]{T} we get that $c_p=1+\text{number of roots of $P(T)$ in $k_p$}$, where $P(T)=T^3+11664.$ The Tamagawa number at $p$ when the reduction type is $I_0^*$ can be $1$, $2$ or $4$ (see Table \ref{table1}). As opposed to the aforementioned two curves, all of those cases are possible here, as noted in Table \ref{table4}.

Furthermore, the Galois group of the polynomial $P$ is $S_3$, hence non-abelian, which means that we have no straightforward description in terms of congruences of the primes $p$ for which there is a root modulo $p$ \cite[pp.576]{W}. However, using Frobenius' density theorem \cite[pp.32]{LS} we know that the density of all primes $p$ such that $P$ remains irreducible modulo $p$ is $\frac{1}{3}.$ That means that we have infinitely many primes $p$ such that $c_p=1.$ It remains to see that $c_E=1$ as well.

The prime $3$ is the only prime of bad reduction for $E$ and the reduction type of $E$ at $3$ is $II.$ The reduction type at $3$ of $E^d$ when $3\nmid d$ stays the same (see Table \ref{tab2}) and hence the Tamagawa number at $3$ stays $1$. As for the prime $p=2$, by using the fact that $\Q_p^\times/\left( \Q_p^\times\right)^2=\langle -1,2,5 \rangle$, we explicitly compute all possible reduction types of quadratic twists at $p=2$ and conclude that for every twist of those curves $c_2=1.$

Therefore, we have that $c_{E^d}=c_d=1$ for $d$ in the set of primes for which $P \text{ mod } p$ stays irreducible.
\end{proof}

\begin{example}
 Denote by $E$ the curve \lmfdbec{27a3}{27.a4}. We have that $c_{E}=c_3=1$, where $3$ is the only prime of bad reduction. As noted in the proof of Proposition \ref{6izo}, if $p$ is a prime of good reduction for $E$, then by Tate's algorithm \cite[\S 8. Case 6)]{T} we get that the Tamagawa number of $E^d$ at $p$ such that $p\mid d$ is $c_p=1+\text{number of roots of $P(T)$ in $k_p$}$, where $P(T)=T^3+11664.$ The curve $E^d$ has reduction type $I_0^*$ at $p$ and the Tamagawa number at $p$ can be $1$, $2$ or $4$ (see Table \ref{table1}), depending on the number of roots of $P \text{ mod } p$. All of those cases are possible here, and they are presented in Table \ref{table4}.
 
 \begin{table}[h]
 	\begin{tabular}{|c|c|c|c|c|}
 		\hline
 		twist of \lmfdbec{27a3}{27.a4} by $d$  & curve & reduction type at $d$ & $c_d$ & $c_E$ \\ \hline
 		$d=7$ & \lmfdbec{21168cx1}{21168.bv4} & $I_0^*$ & $1$&$1$  \\  
 		$d=5$ & \lmfdbec{675a1}{675.e4} & $I_0^*$ & $2$  & $2$  \\ 
 		$d=31$ & \lmfdbec{415152ci1}{415152.ci4} & $I_0^*$ & $4$ & $4$ \\  \hline
 	\end{tabular}
 	\vspace{0,2cm}
 	\caption{some twists of \lmfdbec{27a3}{27.a4} and their Tamagawa numbers}\label{table4}
 \end{table}
All computations made in this example can be verified using the accompanying \href{https://web.math.pmf.unizg.hr/~atrbovi/magma/magma3/ex3.6.txt}{Magma code}.


	
	

\end{example}

\begin{proposition}\label{ostalo}
	Let $E$ be an elliptic curve over $\Q$ with an $n-$isogeny, $n\in\{14,17,19,37,43,67,163\}$. Then $2|c_E.$
\end{proposition}

\begin{proof}
	For each value of $n$, from \cite[Table 4]{ALR} we took all the possible $j-$invariants. They can be found in \Cref{tab1} in the second column. In the third column we have a Cremona label of one of the curves in the class of twists represented by each $j-$invariant. For each of those curves in the fourth column we have a prime of bad reduction of type $III.$ That reduction  can only change to $III^*,$ and vice versa, after twisting, as we see in \Cref{tab2}. \Cref{table1} tells us that the Tamagawa number at primes of reduction type $III$ and $III^*$ is always 2, so the claim follows.
	
\end{proof}
\newpage
	\begin{table}[h!]
		\centering
		
		\begin{tabular}{|r|ccc|}\hline
			\multicolumn{1}{|c|}{$\mathit{n}$} & $j-$invariant & Cremona label & \begin{tabular}[c]{@{}c@{}}bad prime with \\ reduction type $III$\end{tabular} \\ \hline
			14                     &       $-3^3\cdot 5^3$    &        \lmfdbec{49a1}{49.a4}       &               7                 \\
			&       $3^3\cdot 5^3\cdot 17^3$      & \lmfdbec{49a2}{49.a3}       &                 7               \\ \hline
			17                     &      $-\frac{17^2\cdot 101^3}{2}$       &        \lmfdbec{14450p1}{14450.b2}       &               5                 \\
			&      $-\frac{17\cdot 373^3}{2^{17}}$       &        \lmfdbec{14450p2}{14450.b1}       &                5                \\ \hline
			19                     &        $-2^{15}\cdot 3^3$     &              \lmfdbec{361a1}{361.a2} &             19                   \\ \hline
			37                     &      $-7\cdot 11^3$      &       \lmfdbec{1225h1}{1225.b2}        &          5                      \\
			&    $-7\cdot 137^3\cdot 2083^3$      &       \lmfdbec{1225h2}{1225.b1}    &           5                     \\ \hline
			43                     &      $-2^{18}\cdot 3^3\cdot 5^3$       &       \lmfdbec{1849a1}{1849.b2}       &              43                  \\ \hline
			67                     &       $-2^{15}\cdot 3^3\cdot 5^3\cdot 11^3$      &      \lmfdbec{4489a1}{4489.b2}         &               67                 \\ \hline
			163                    &       $-2^{18}\cdot 3^3\cdot 5^3\cdot 23^3\cdot 29^3$      &      \lmfdbec{26569a1}{26569.a2}        &            163     \\              
		\hline
	\end{tabular}
		\vspace{0,2cm}
		\caption{$j-$invariants of the curves $X_0(n)$; their Cremona labels are representatives in the class of twists of least conductor with reduction type $III$ at some prime}\label{tab1}
	\end{table}

\subsection*{Acknowledgments}
The author would like to thank Filip Najman for comments and helpful discussions.

\bibliographystyle{unsrt}

\end{document}